\documentclass[10pt]{amsart}
\usepackage[margin=1.5in]{geometry}
\usepackage[colorlinks,citecolor=cyan,linkcolor=teal]{hyperref}
\author[]{James Farre}

\author[]{Beatrice Pozzetti}

\author[]{Gabriele Viaggi}

\newcommand{\Addresses}{{
  \bigskip
  \footnotesize  
  \noindent James Farre, \textsc{Max Planck Institute for Mathematics in the Sciences, Leipzig}\par\nopagebreak
  \textit{E-mail address}: \texttt{james.farre@mis.mpg.de}
  
  \noindent Beatrice Pozzetti, \textsc{Mathematical Institute, Heidelberg University, Heidelberg}\par\nopagebreak
  \textit{E-mail address}: \texttt{pozzetti@mathi.uni-heidelberg.de}
  
  \noindent Gabriele Viaggi, \textsc{Mathematical Institute, Sapienza University of Rome, Rome}\par\nopagebreak
  \textit{E-mail address}: \texttt{gabriele.viaggi@uniroma1.it}
  }}

\date{}

\subjclass[]{}

\usepackage{amsmath}
\usepackage{amsthm}
\usepackage{amssymb}
\usepackage{enumerate}
\usepackage[all,cmtip]{xy}
\usepackage{todonotes}

\newcommand{\C}{\mathbb C}
\renewcommand{\P}{\mathbb P}
\newcommand{\R}{\mathbb R}
\newcommand{\calF}{\mathcal F}

\newcommand{\PGL}{{\rm PGL}}

\newcommand{\Gr}{{\rm Gr}}
\newcommand{\bord}{{\partial}}
\newcommand{\G}{\Gamma}
\newcommand{\HH}{\mathbb H}
\newcommand{\g}{\gamma}

\newcommand{\cal}{\mathcal }

\newcommand{\bP}{\mathbb {P}}

\newcommand{\deG}{\partial\Gamma}

\renewcommand{\H}{{\mathbb H}}

\newcommand{\inverse}{^{-1}}
\newcommand{\CP}{\mathbb{CP}^1}
\newcommand{\CPd}{\mathbb{CP}^{d-1}}
\newcommand{\RP}{\mathbb{RP}^1 }
\newcommand{\dH}{\partial \mb{H}^2}

\newcommand{\bH}{{\mb{H}}}

\DeclareMathOperator{\Homeo}{Homeo}

\newcommand{\calC}{\mathcal{C}}

\newcommand{\Stab}{{\rm Stab}}
\newcommand{\Id}{{\rm Id}}
\newcommand{\detG}{\partial^{(2)}\G}
\newcommand{\deH}{\partial\HH}
\newcommand{\cdAR}{{\rm confdim_{AR}}}

\newcommand\quotient[2]{
	\mathchoice
	{% \displaystyle
		\text{\raise.5ex\hbox{$#1$}\big/\lower.5ex\hbox{$#2$}}%
	}
	{% \textstyle
		\text{\raise.25ex\hbox{$#1$}\big/\lower.25ex\hbox{$#2$}}%
	}
	{% \scriptstyle
		#1\,/\,#2
	}
	{% \scriptscriptstyle  
		#1\,/\,#2
	}
}

\newcommand{\mc}[1]{
\mathcal{#1}
}
\newcommand{\mb}[1]{
\mathbb{#1}
}

\newcommand{\PSL}{{\rm PSL}}

\newtheorem*{thm*}{Theorem}
\newtheorem*{cor*}{Corollary}
\newtheorem*{prop*}{Proposition}

\newtheorem{thmA}{Theorem}

\newtheorem{thm}{Theorem}[section]

\newtheorem{cor}[thm]{Corollary}
\newtheorem{lemma}[thm]{Lemma}
\newtheorem{lem}[thm]{Lemma}
\newtheorem{prop}[thm]{Proposition}

\theoremstyle{definition}

\newtheorem*{defi*}{Definition}
\newtheorem{dfn}[thm]{Definition}

\newtheorem{remark}[thm]{Remark}

\theoremstyle{remark}

\newcommand{\nocontentsline}[3]{}
\newcommand{\tocless}[2]{\bgroup\let\addcontentsline=\nocontentsline#1{#2}\egroup}
\title[Topological and geometric restrictions]{Topological and geometric restrictions on hyperconvex representations}

\begin{document}

\begin{abstract}
We study the geometry of hyperconvex representations of hyperbolic groups in ${\rm PSL}(d,\mb{C})$ and establish two structural results:  a group admitting a hyperconvex representation is virtually isomorphic to a Kleinian group, and its hyperconvex limit set in the appropriate flag manifold has Hausdorff dimension strictly smaller than $2$.
\end{abstract}

\maketitle
\vspace{-1cm}
\section{Introduction}

A {\em Kleinian group} is a discrete subgroup $\Gamma<{\rm PSL}(2,\mb{C})$. 
The study of these groups is a rich topic that interlaces 3-dimensional hyperbolic geometry, conformal dynamics, and Teichmüller theory. 
While the quotient of $\H^3$ by a Kleinian group is a complete hyperbolic $3$-orbifold, 
the action on $\mb{CP}^1$ splits into two parts:
\begin{enumerate}[(i)]
    \item{On the {\em limit set} $\Lambda\subset\mb{CP}^1$, the set of accumulation points of the orbits, it is chaotic (ergodic). Typically, this is a fractal object.}
    \item{On the complement $\Omega=\mb{CP}^1-\Lambda$, the {\em domain of discontinuity}, it is properly discontinuous  and gives us a Riemann surface  $\Sigma=\Omega/\Gamma$, possibly  with orbifold points.}
\end{enumerate}

An important class of Kleinian groups are the so-called {\em convex-cocompact} groups, which are discrete subgroups $\Gamma<{\rm PSL}(2,\mb{C})$ leaving invariant a non-empty closed convex set in $\H^3$ on which the action is co-compact. The domain of discontinuity of convex-cocompact groups $\Gamma$ that are not co-compact is always non-empty and the limit set is homeomorphic to the Gromov boundary $\deG$ of $\Gamma$.

\emph{Anosov subgroups} (Definition \ref{def: Anosov}) were introduced by Labourie \cite{L06} and Guichard-Wienhard \cite{GW:Anosov_DOD} and are now generally accepted as the correct generalization of convex-cocompactness to higher rank. 
\emph{Hyperconvex subgroups} (Definition \ref{d.hyp}) were singled out  in \cite{PSW:HDim_hyperconvex} among Anosov subgroups as having features in common with Kleinian groups. 
In this paper and in \cite{FPV2}, we begin investigating which phenomena of the theory of Kleinian groups persist in the class of  hyperconvex subgroups of ${\rm PSL}(d,\mb{C})$.

This work is focused on what restrictions  are imposed on a group admitting a hyperconvex representation and on the geometry of its \emph{limit set}.
If a finitely generated group $\Gamma$ admits a $k$-Anosov representation $\rho: \Gamma \to \PSL(d,\C)$, then $\Gamma$ is hyperbolic and admits a continuous, equivariant boundary map $\xi^k: \partial \Gamma \to \Gr_k(\mb C^d)$.
The image, its \emph{$k$-th limit set}, $\Lambda^k_\rho = \xi^k(\partial \Gamma)$ is the smallest closed set in $\Gr_k(\C^d)$ preserved by $\rho(\Gamma)$.

By intersecting linear subspaces of $\C^d$ corresponding to boundary maps in suitable dimensions, we construct in \S\ref{sec:bundle}, a $\CP$-bundle over $\deG$ equipped with an action by $\CP$-bundle automorphisms covering the topological action of $\G$ on $\deG$.  Using $k$-hyperconvexity, for each $z\in \deG$,  we obtain an \emph{embedding} of $\Lambda_\rho^k$ into the fiber over every point, and these embeddings vary continuously. This $\deG$ bundle over $\deG$ is called the \emph{foliated $k$-limit set} of $\rho$.  It is the central object of study and allows us to import tools from 2-dimensional quasi-conformal analysis and 3-dimensional hyperbolic geometry to the world of  Anosov representations, constituting the main novelty of this work.

In this paper, we show that a group $\Gamma$ admitting a hyperconvex representation $\rho:\Gamma\to{\rm PSL}(d,\mb{C})$ has restrictions both topologically (Theorem \ref{thmINTRO:Kleinan}) and on the geometry of the limit sets  (Theorem \ref{thmINTRO:Hausdorff}). 
In forthcoming work \cite{FPV2}, we will consider surface groups and generalizations in ${\rm PSL}(d,\mb{C})$ of {\em quasi-Fuchsian representations}. 
There we will develop more in-depth the relation with Teichmüller theory by way of replacing $\Omega$, the domain of discontinuity for a Kleianian group, with a foliated analogue, the complement of the foliated limit set in the $\CP$-bundle over $\deG$.  

Our first result establishes a direct connection with  classical Kleinian groups.
\begin{thmA}\label{thmINTRO:Kleinan}
A hyperbolic group $\Gamma$ that admits a {\rm hyperconvex} representation in ${\rm PSL}(d,\mb{C})$ is virtually isomorphic to a convex-cocompact Kleinian group.
\end{thmA}

\noindent Recall that groups $G,G'$ are virtually isomorphic if there are finite index subgroups $H<G,H'<G'$ and finite normal subgroups $F\triangleleft H,F'\triangleleft H'$ such that $H/F\simeq H'/F'$. 
The question of finding topological  restrictions on groups that admit { Anosov representations} in higher rank Lie groups has been investigated, e.g., by Canary--Tsouvalas \cite{CanaryTsouvalas}, Dey \cite{Dey},  Tsouvalas \cite{Tsouvalas} for ${\rm SL}(d,\mb{R})$,  Dey--Greenberg--Riestenberg \cite{DGR} for ${\rm Sp}(2d,\mb{R})$, and Pozzetti--Tsouvalas \cite{PTsouvalas} for ${\rm Sp}(2d,\mb{C})$.

It was shown in \cite{PSW:HDim_hyperconvex}, that if $\Gamma$ is a hyperbolic group that admits a hyperconvex representation in ${\rm PSL}(d,\mb{C})$, then the Gromov boundary $\partial\Gamma$ embeds in the 2-sphere. However, in general, it is wide open whether a hyperbolic group whose boundary embeds in the 2-sphere must be virtually isomorphic to a Kleinian group: this is Cannon's conjecture \cite{Cannon} if the boundary is homeomorphic to the 2-sphere, and a conjecture of Kapovich and Kleiner \cite{KK00} if the boundary is homeomorphic to a Sierpinski carpet.
Recall that a Sierpinski carpet is the complement in $\CP$ of the interiors of a family of pairwise disjoint closed disks that are dense in $\CP$; Whyburn proved that any two such topological spaces are homeomorphic \cite{Whyburn}.

Our approach to prove that $\Gamma$ is Kleinian relies on Sullivan's result \cite{Sul81} that a uniformly quasi-conformal subgroup of $\Homeo(\CP)$ is quasi-conformally conjugated to a subgroup of $\PSL_2(\C)$. 
In the case that $\partial \Gamma$ is a $2$-sphere, we use hyperconvexity to build an invariant quasi-conformal structure on $\partial \Gamma$ to which Sullivan's result applies (see \S\ref{ss.sphere}).
Thanks to a result of Ha\"issinksy \cite{Ha15}, the general case reduces to proving the result for groups with Sierpinski carpet boundary.
In \S\ref{ss.Sierpinski} we deal with this case and use hyperconvexity to find quasi-symmetric actions of the stabilizer of each peripheral circle of the Sierpinski carpet that all piece together nicely.
A result of Markovic \cite{Markovic} allows us to extend these quasi-symmetric actions on peripheral circles to quasi-disks that they bound in $\CP$ so that we may then apply Sullivan's result, yet again. It would be interesting to understand if our techniques can be generalized to obtain topological restrictions on   (1,1,2)-hypertransverse groups, introduced by Canary-Zimmer-Zhang \cite{CZZ}.

\medskip
Seminal work of Sullivan \cite{Sul84} and Tukia \cite{Tuk84} shows that being convex-cocompact tames the complexity of the limit set: the Hausdorff dimension of the limit set of a geometrically finite group of M\"obius transformations is strictly less than the dimension of the ambient sphere unless the group is a lattice.  
We show that in  higher rank  {\em hyperconvexity} also prevents the limit set from degenerating too much.

\begin{thmA}\label{thmINTRO:Hausdorff}
Let $\rho:\Gamma\to{\rm PSL}(d,\mb{C})$ be  $k$-hyperconvex. 
Then the $2$-dimensional Hausdorff measure of the $(d-k)$-th limit set $\Lambda_\rho^{d-k}\subset \Gr_{d-k}(\mathbb C^d)$ is $0$ unless $\Gamma$ is virtually isomorphic to a uniform lattice in ${\rm PSL}(2,\mb{C})$.
If $\deG$ is homeomorphic to a {\rm Sierpinski carpet} or a circle, then furthermore  
\[
{\rm Hff}(\Lambda_\rho^{d-k})<2.
\]
\end{thmA}

We actually prove a stronger result about the Hausdorff dimension of the  limit set $\Lambda$ of $\rho$ in a finer partial flag manifold if $\partial\Gamma$ is a Sierpinski carpet or circle and $\rho$ satisfies additional hyperconvexity hypotheses; see Theorem \ref{p.Hffub}.  

The boundaries of fundamental groups of compact hyperbolic 3-manifolds with totally geodesic boundary are Sierpinski carpets. The assumption of Theorem \ref{thmINTRO:Hausdorff} can be rephrased as requiring that $\G$ is the fundamental group of an irreducible atoroidal acylindrical 3-manifold with non-empty incompressible boundary. The restriction in the second part of Theorem \ref{thmINTRO:Hausdorff} to groups whose boundary is a Sierpinski carpets guarantees that, in each fiber of our $\CP$-bundle over $\deG$, the limit set is quasi-conformally conjugated to a(ny) Kleinian realisation. 

Pozzetti-Sambarino-Wienhard already proved that  ${\rm Hff}(\Lambda^{d-k}_\rho)\le 2$ \cite{PSW:HDim_hyperconvex}. However the fact that 2-dimensional Hausdorff measure of $\Lambda^{d-k}_\rho$ vanishes is important, among other things, in view of results of Dumas-Sanders on the complex geometry of the domain of discontinuity (\cite{DumasSanders} and \cite{DumasSanders2}). 
To give a concrete example, let $\rho:\G\to{\rm PSL}(4,\mb{C})$ be a 2-Anosov representation. A co-compact domain of discontinuity for $\rho$ is given by
\[
\Omega_\rho:=\mb{CP}^3\setminus\bigcup_{t\in\partial\Gamma}{\mb{P}(\xi^2(t))},
\]
this is one of the domains studied by Guichard-Wienhard \cite{GW:Anosov_DOD}. Denoting by $M_\rho:=\Omega_\rho/\rho(\Gamma)$ the compact quotient, in order to obtain the following corollary from \cite[Theorems D and E]{DumasSanders}, one needs to show that the 2-dimensional Hausdorff measure of the limit set $\Lambda^2_\rho\subset \Gr_2(\C^4)$ vanishes, which is the content of Theorem \ref{thmINTRO:Hausdorff} in the presence of hyperconvexity:

\begin{cor}
Let $\rho:\G\to{\rm PSL}(4,\mb{C})$ be 2-hyperconvex. The manifold $M_\rho$ does not admit a Kähler metric and does not admit any non-constant holomorphic map to a hyperbolic Riemann surface.   
\end{cor}

\subsection*{Acknowledgements}
This work was funded through the DFG Emmy Noether project 427903332 of Beatrice Pozzetti. Beatrice Pozzetti acknowledges additional support by the DFG under Germany’s Excellence Strategy EXC-2181/1-390900948.
James Farre acknowledges support from DFG – Project-ID 281071066 – TRR 191. 
\setcounter{tocdepth}{1}
\tableofcontents

\section{The tangent projection}\label{s.projmap}

\subsection{Anosov representations}\label{s.Anosov}
Anosov representations of hyperbolic groups were introduced by Labourie \cite{L06} and Guichard-Wienhard \cite{GW:Anosov_DOD} using dynamical properties of the lift of the geodesic flow on the flat bundle associated to the representation. 
We use instead the  geometric group theoretic definition due to Kapovich-Leeb-Porti \cite{KLP} and Bochi-Potrie-Sambarino \cite{BPS}, based on the following quantities associated to an element $A\in{\rm PSL}(d,\mb{C})$.

\begin{dfn}[Singular values gaps]
Choose a Hermitian inner product on $\mb C^d$.
For an element $A\in{\rm PSL}(d,\mb{C})$ we denote by $\sigma_1(A)\ge\cdots\ge\sigma_d(A)$ the \emph{singular values} of any lift of $A$ to ${\rm SL}(d,\mb{C})$ ordered in decreasing size, namely the eigenvalues of the Hermitian matrix $A^*A$. For $1\le k <d$, the  $k$-th {\em singular value gap} is 
\[\frac{\log \sigma_k(A)-\log\sigma_{k+1}(A)}{2}.\]
\end{dfn}

The singular value gaps can be interpreted as a refinement of the distance between the basepoint in the symmetric space of ${\rm PSL}(d,\mb{C})$ and its image under $A$. Anosov representations are then representations for which the orbit map is strongly undistorted. For a finitely generated group $\Gamma$ we choose a symmetric finite generating set $S$ and denote by $|\cdot|$ the associated word-length. 
\begin{dfn}[Anosov representation]\label{def: Anosov}
For $1\leq k < d$, a representation $\rho:\G\to\PSL(d,\C)$ is \emph{$k$-Anosov} if the $k$-th singular value gap of $\rho(\gamma)$ is bounded below by an affine function of $|\gamma|$, i.e., there exist positive constants $c_1,c_2$ such that 
\[
\log\sigma_{k}(\rho(\g))-\log\sigma_{k+1}(\rho(\g))\geq c_1|\gamma|-c_2
\]
for all $\gamma \in \Gamma$.
\end{dfn}
It is immediate from the definition that a $k$-Anosov representation is also $(d-k)$-Anosov. It was proven in \cite{KLP,BPS} that a group admitting a $k$-Anosov representation is hyperbolic, and any $k$-Anosov representation $\rho$ admits continuous, equivariant, transverse boundary maps $\xi^k:\partial\Gamma\to\Gr_k(\C^d)$, $\xi^{d-k}:\partial\Gamma\to\Gr_{d-k}(\C^d)$. The latter means that for every $x\neq y$, the subspaces $\xi^k(x)$ and $\xi^{d-k}(y)$ are in direct sum. Following the notation introduced in \cite{PSW:HDim_hyperconvex} we will sometimes just write for a point $x\in\partial\Gamma$
\[
x^k:=\xi^k(x)\in\Gr_k(\C^d).
\]
If $\rho$ is $k_1$-Anosov and $k_2$-Anosov with $k_1<k_2$ then for every $x\in\deG$ it holds $x^{k_1}<x^{k_2}$.

The \emph{limit set} of the $\rho$ action on $\Gr_k(\C^d)$ is the smallest closed $\rho(\G)$-invariant set, and  boundary map $\xi^k$ parametrizes this set.  
\begin{dfn}[Limit set]
 We denote by \[\Lambda^k_\rho:=\xi^k(\deG)\subset\Gr_k(\C^d)\] the \emph{$k$-th limit set} of the $k$-Anosov representation $\rho:\G\to\PSL(d,\C)$.
\end{dfn}

\subsection{Hyperconvexity and tangent projections}
\label{sec:hyperconvex}

Among Anosov representations into ${\rm PSL}(d,\mb{C})$ we will focus on the more restricted class given by hyperconvex representations.  In the following, any representation into $\PSL(d,\C)$ is considered to be both $0$ and $d$-Anosov with the constant boundary maps $\xi^0(x)=\{0\}$ and $\xi^d(x)=\C^d$ for all $x\in\deG$.

\begin{dfn}[$k$-hyperconvex]
\label{d.hyp}

A representation $\rho: \Gamma \to{\rm PSL}(d,\mb{C})$ is {\em $k$-hyperconvex} for $1\le k\le d-1$ if it is $\{k-1,k+1,d-k\}$-Anosov, and, for all distinct triples $x, y, z\in \partial \Gamma$, we have 
\begin{equation}\label{e.hyp}\left((x^{d-k}\cap z^{k+1})+z^{k-1}\right)\cap\left((y^{d-k}\cap z^{k+1})+z^{k-1}\right) = z^{k-1}.
\end{equation}
\end{dfn}

\begin{remark}
This is the dual notion to property $H_k$ as defined in \cite[Definition 6.1]{PSW:HDim_hyperconvex}, which would require that the sum $(x^k\cap z^{d-k+1})+(y^k\cap z^{d-k+1})+z^{d-k-1}$ is direct. 
More specifically, a representation $\rho$ is  $k$-hyperconvex according to Definition \ref{d.hyp} if and only if its contragradient representation $\rho^*$ satisfies property $H_k$ according to \cite{PSW:HDim_hyperconvex}. 
\end{remark}

Hyperconvexity is an open condition on representations \cite[Proposition 6.2]{PSW:HDim_hyperconvex}.

Let $\rho:\Gamma\to{\rm PSL}(d,\mb{C})$ be a $k$-hyperconvex representation. For each $z\in \partial \Gamma$, we project the limit set $\Lambda^{d-k}_\rho\subset\Gr_{d-k}(\C^d)$ to the projective line $\bP(z^{k+1}/z^{k-1})$. More specifically we study the map $\xi_z^k:\partial \Gamma\to\bP(z^{k+1}/z^{k-1})$

\[
\xi_z^k(x):=\left\{
\begin{array}{l l}
{[x^{d-k}\cap z^{k+1}]} &\text{if $x\neq z$,}\\
{[z^k]} &\text{if $x=z$.}\\
\end{array}
\right.
\]

The Anosov property ensures that these maps are well defined. The hyperconvexity property guarantees that these maps are furthermore injective and continuous. This is the content of the next proposition, which was proven, for representations satisfying property $H_k$ in \cite[Proposition 4.9]{BeyP1} building on ideas from \cite[Proposition 6.7]{PSW:HDim_hyperconvex}.

\begin{prop}
\label{prop:map xi}
Let $\rho:\G\to\PSL(d,\C)$ be $k$-hyperconvex. For every $z\in\partial\Gamma$ the map $\xi_z^k(\cdot)$ is well defined, continuous and injective.
\end{prop}

For the proof we will need the following proposition, in which we denote by $\detG$ the set of pairs of distinct points in $\deG$, and for $x^k\in\Gr_k(\C^d)$, $y^{d-k}\in\Gr_{d-k}(\C^d)$ we write $x^k\pitchfork y^{d-k}$ to mean $x^k\cap y^{d-k}=\{0\}$:
\begin{prop}[{\cite[Proposition 2.13]{BeyP1}}]\label{p.BeyP1}
    Let $\rho:\G\to\PGL(d,\C)$ be $k$-Anosov.  Suppose $F:\detG\to\Gr_k(\C^d)$ is a continuous,  $\rho$-equivariant, and  satisfies that for every pairwise distinct triple $(x,y,z)$, $F(x,z)\pitchfork y^{d-k}$. Then $$\lim_{x\to z} F(x,z)=z^k.$$
\end{prop}

\begin{proof}[Proof of Proposition \ref{prop:map xi}]
We consider the map
$$\begin{array}{cccc}
F:&\detG&\to&\Gr_k(\C^d)\\
&(x,z)&\mapsto& (x^{d-k}\cap z^{k+1})+z^{k-1}\end{array}$$
Since the representation is $k$-Anosov, for every distinct pair $x,z\in\deG$, the subspace $x^{d-k}$ is transverse to $z^k$, as a result the map is well defined. Since the representation is $k$-hyperconvex it furthermore follows that if $(x,y,z)\in\deG^3$ is pairwise distinct, then $F(x,z)\pitchfork y^{d-k}$: since by definition $F(x,z)<z^{k+1}$ it holds 
$$F(x,z)\cap y^{d-k}< F(x,z)\cap ((y^{d-k}\cap z^{k+1})+z^{k-1}).$$
Since $\rho$ is $k$-hyperconvex, the latter subspace is equal to $z^{k-1}$, and so the former is contained in $z^{k-1}$.  Since $y^{d-k}$ is transverse to $z^k$, $F(x,z)\cap y^{d-k}$ is trivial.

Since $F(x,z)$ is clearly $\rho$-equivariant and continuous, it follows from Proposition \ref{p.BeyP1} that we can use $F(x,z)$ to define the continuous map 
$F_z:\deG\to\Gr_k(\C^d)$ 
\[
F_z(x):=\left\{
\begin{array}{l l}
{(x^{d-k}\cap z^{k+1})+z^{k-1}} &\text{if $x\neq z$,}\\
{z^k} &\text{if $x=z$.}\\
\end{array}
\right.
\]
Since for every $x$, 
$z^{k-1}<F_z(x)<z^{k+1}$, we obtain, by projection, that $\xi^k_z$ is well defined and continuous.

The map $\xi^k_z$ is injective: for every $x\neq z$, $\xi^k_z(x)\neq \xi^k_z(z)$ since the representation is $k$-Anosov and thus in particular $x^{d-k}\pitchfork z^k$;  for every $x,y$ distinct and distinct from $z$, that $\xi^k_z(x)\neq \xi^k_z(y)$ is a restatement of Equation \eqref{e.hyp}.
\end{proof}

If $\rho:\G\to\PSL(d,\C)$ is $k$-hyperconvex, we set, for every point $z\in\deG$ 
$$\Lambda_z=\Lambda^k_z:=\xi_z^k(\deG)\subset \P(z^{k+1}/z^{k-1}).$$

\subsection{Induced conformal actions}
\label{sec:bundle}
Consider the partial flag manifolds
\begin{align*}
&\calF_{k-1,k,k+1}=\{(U,V,W)\in{\rm Gr}_{k-1}(\C)\times{\rm Gr}_{k}(\C)\times{\rm Gr}_{k+1}(\C)\left|\;U< V< W\right.\}\\
&\calF_{k-1,k+1}=\{(U,W)\in{\rm Gr}_{k-1}(\C)\times{\rm Gr}_{k+1}(\C)\left|\;U< W\right.\},
\end{align*}
so that the  continuous surjection
\[
\calF_{k-1,k,k+1}\to \calF_{k-1,k+1}
\]  
is a fiber bundle with fibers isomorphic to $\mb{CP}^1$.
It can be naturally identified with the canonical fiber bundle $\mathcal B\to\calF_{k-1,k+1}$ with fiber $\mb{P}(W/U)$ over the pair $(U,W)\in\calF_{k-1,k+1}$.

Use the map $t\in \deG \mapsto (t^{k-1}, t^{k+1}) \in \mathcal {F}_{k-1,k+1}$ to obtain a $\CP$-bundle $\mathcal B_\rho^k \to \deG$ by pullback.
For each $\gamma \in \Gamma$ and $t\in \deG$, the linear map $\rho(\gamma) \in \mathrm{PSL}(d,\C)$ induces a projective equivalence 
\begin{equation}\label{eqn: mobius cocycle}
    \rho^k(\g,t):\mb{P}(t^{k+1}/t^{k-1})\to\mb{P}((\g t)^{k+1}/(\g t)^{k-1}).
\end{equation}
Thus $\rho$ induces a $\G$ action on $\mathcal B_\rho^k$ by $\CP$-bundle automorphisms covering the natural action of $\G$ on $\deG$.
By construction, $\Lambda_z$ is contained in the fiber over $z$ and the $\G$-action via $\rho^k(\gamma, z)$ takes $\Lambda_z^k$ conformally to $\Lambda_{\gamma z}^k$.

\begin{lem}\label{lem: double boundary continuous into bundle}
    The map 
    \[(x,t)\in \deG^2 \mapsto (\xi_t^k(x), t) \in \mathcal B_\rho^k\]
    is a homeomorphism onto its image.
\end{lem}
\begin{proof}
    Continuity and injectivity follow directly from the proof of Proposition \ref{prop:map xi} and construction of $\mathcal B_\rho^k$.  The result follows as $\deG^2$ and $\mathcal B_\rho^k$ are both compact and Hausdorff. 
\end{proof}

\begin{dfn}[Foliated limit set]\label{def: double limit set}
    The {\emph{foliated limit set}} is defined to be 
    the image of the map from Lemma \ref{lem: double boundary continuous into bundle}.
\end{dfn}

\medskip

We can trivialize $\mathcal B_\rho^k \to \deG$ by finding three continuous sections in fiber-wise general position.
Indeed, fix a triple of pairwise distinct points $x,y,z\in \deG$, and define 
\[L: \mathcal B_\rho^k \to \CP\times \deG\] 
fiber-wise as the bi-holomorphism 
\[(\xi_t(x), \xi_t(y),\xi_t(z)) \mapsto (0,1,\infty)\]
in the fiber over $t$.
Then $L$ is a isomorphism of $\CP$-bundles trivializing $\mathcal B_\rho^k$. 

Conjugating the $\G$ action \eqref{eqn: mobius cocycle} on $\mathcal B_\rho^k$  by $L$ defines an action of $\G$ on $\CP\times \deG$ by $\CP$-bundle automorphisms, hence a continuous map
\begin{equation}\label{eqn: trivial bundle cocycle}
    \rho_L^k: \G \times \deG \to \PSL(2,\C)
\end{equation}

satisfying the cocycle property
\[\rho_L^k(\alpha\beta, t) = \rho_L^k(\alpha, \beta t)\rho_L^k(\beta, t),\]
for all $\alpha, \beta\in \G$ and $t\in \deG$.
Explicitly, the $\G$ action on $\CP \times \deG$ is given,  for $(z,t) \in \CP\times \deG$ and $\gamma\in \Gamma$, by 
\begin{equation}\label{eqn: conformal bundle action}
    \gamma \in \Gamma : (z,t)\mapsto\left(\rho_L^k(\gamma,t)z,\gamma t\right).
\end{equation}

Denote by $L_t$ the restriction of $L$ to the fiber over $t$.  

\begin{lem}\label{l.conti}
Let $\rho:\Gamma\to{\rm PSL}(d,\mb{C})$ be $k$-hyperconvex. Then the sets $L_t(\Lambda_t^k)$ vary continuously with $t\in\partial\Gamma$ in the Hausdorff topology on closed subsets of $\CP$.
\end{lem}

\begin{proof}
    From Lemma \ref{lem: double boundary continuous into bundle} and continuity of $L$, for each $x\in \deG$ the map $t\mapsto L_t(\xi_t^k(x)) \in \CP$ is continuous.  
    The result follows.
\end{proof}

Note that Equation \eqref{eqn: conformal bundle action} also defines an action of $\G$ on the trivial bundle $\HH^3\times \deG$ which is isometric on the fibers and covers the natural action of $\G$ on $\deG$.

\subsection{Wedge products and $k$-hyperconvexity}
We conclude the section by discussing the link between $k$-hyperconvexity and $1$-hyperconvexity of the composition of a representation with the $k$-th exterior power. This will be needed in the last section of the paper. 
\begin{prop}[{\cite[Proposition 8.11]{PSW:HDim_hyperconvex}}]
\label{prop:j to 1i}
If $\rho:\Gamma\to{\rm PSL}(d,\mb{C})$ is $k$-hyperconvex, then $\wedge^k\rho:\Gamma\to{\rm PSL}(\wedge^k\mb{C}^d)$ is 1-hyperconvex. Furthermore, if $\Lambda_\rho^k\subset{\rm Gr}_k(\mb{C}^d)$ is the $k$-th limit set of $\rho$, then \[\wedge^k\Lambda_\rho^k=\{\wedge^k\xi^k(t)\left|\,t\in\partial\Gamma\right.\}\subset\mathbb P(\wedge^k\mb{C}^d)\] 
is the projective limit set of $\wedge^k\rho$.
\end{prop}
\begin{proof}
We include the proof for completeness. Let us first verify that if $\rho$ is $\{k-1,k,k+1\}$-Anosov, then $\wedge^k\rho$ is $\{1,2\}$-Anosov. Indeed if we fix on $\wedge^k\C^d$ the Hermitian product induced from the chosen Hermitian product on $\C^d$, and given $\g\in \G$ we denote by $\sigma_i$ the singular values of $\rho(\g)$ and by $\sigma_i^k$ the singular values of $\wedge^k\rho(\g)$ we have $\sigma_1^k = \sigma_1 ... \sigma_k$, $\sigma_2^{k} = \sigma_1...  \sigma_{k-1} \sigma_{k+1}$ and  $\sigma_3^k=\max\{\sigma_1 ... \sigma_{k-2} \sigma_k \sigma_{k+1}, \sigma_1 ...\sigma_{k-1}\sigma_{k+2}\}$.  In particualr, the first singular value gap of $\wedge^k\rho(\g)$ is $1/2\log(\sigma_{k+1}/\sigma_{k}$) and the second is $\min \{1/2\log(\sigma_{k}/\sigma_{k-1}) ,1/2\log(\sigma_{k+2}/\sigma_{k+1})  \}$, which shows the first claim.

Next we discuss how to obtain explicit boundary maps for $\wedge^k\rho$. 
As above, we denote by $\calF_{k-1,k+1}$ the partial flag manifold consisting of nested subspaces of the given dimension, and consider the algebraic maps
$$\begin{array}{cccl}
\Xi:&\calF_{k-1,k+1}&\to&\Gr_2(\wedge^k\C)\\
&(\langle y_1,\ldots, y_{k-1}\rangle,Y^{k+1})&\mapsto &\{ y_1\wedge\ldots\wedge y_{k-1}\wedge y|\, y\in Y^{k+1}\}
\end{array}$$
and 
$$\begin{array}{cccl}
\Psi:&\Gr_{d-k}(\C)&\to&\Gr_{N-1}(\wedge^k\C)\\
&\langle y_1,\ldots, y_{d-k}\rangle&\mapsto &\{\langle x_1,\ldots, x_{k}\rangle|\, x_1\wedge\ldots\wedge x_k\wedge y_1\ldots\wedge y_{d-k}=0\},
\end{array}$$
where $N$ denotes $\dim_\C(\wedge^k\C^d)$.
It is easy to verify that the maps are well defined, and that the compositions  $\Xi\circ(\xi^{k-1},\xi^{k+1}):\deG\to \Gr_2(\wedge^k\C)$ and $\Psi\circ\xi^{d-k}:\deG\to \Gr_{N-1}(\wedge^k\C)$ are the Anosov boundary maps for $\wedge^k\rho$.

In order to verify that $\wedge^k\rho$ is 1-hyperconvex we need to check that for every pairwise distinct triple $x,y,z\in\deG$ it holds 
$$\left(\Psi\circ\xi^{d-k}(x)\right)\cap \left(\Psi\circ\xi^{d-k}(y)\right)\cap\left(\Xi\circ(\xi^{k-1},\xi^{k+1})(z)\right)=\{0\}.$$
Assume that this is not the case, then we can find a non-zero vector $v$ in the intersection. Since $v$ belongs to  $\Xi\circ(\xi^{k-1},\xi^{k+1})(z)$, it has the form 
$v=z_1\wedge\ldots\wedge z_{k-1}\wedge w$ for $z_1,\ldots, z_{k-1}$ a basis of $\xi^{k-1}(z)$ and $w\in \xi^{k+1}(z)\setminus \xi^{k-1}(z)$.  Since $x^{d-k}$ is transverse to $z^{k-1}$ that $v$ belongs to $\Psi\circ\xi^{d-k}(x)$ implies that $w\in x^{d-k}\oplus z^{k-1}$, and similarly that $v$ belongs to $\Psi\circ\xi^{d-k}(y)$ implies that $w\in y^{d-k}\oplus z^{k-1}$. This contradicts $k$-hyperconvexity of $\rho$.
\end{proof}

The bundles constructed in the previous subsection are compatible with  taking wedge powers as in Proposition \ref{prop:j to 1i}.

\begin{prop}\label{prop: bundles equivalent wedge}
    Suppose $\rho: \G \to \PSL(d,\C)$ is $k$-hyperconvex.
    Then there is a natural $\Gamma$-equivariant $\CP$-bundle isomorphism 
    \[\mathcal B^k_{\rho} \cong \mathcal B^1_{\wedge^k\rho}\]
    over the identity $\deG\to \deG$ compatible with the canonical parameterizations of the foliated limit sets.
\end{prop}

\begin{proof}
Recall that the fiber over a point $z$ in $\mathcal B^k_{\rho}$ is the projective space $\P(z^{k+1}_\rho/z^{k-1}_\rho)$. We then have a natural map
$$\begin{array}{cccl}
\Phi:&\mathcal B^k_{\rho}&\to&\P(\wedge^k\C)\\
&([v],z)&\mapsto & [z_1\wedge\ldots\wedge z_{k-1}\wedge v],
\end{array}$$
where $z^{k-1}=\langle z_1,\ldots, z_{k-1}\rangle$. It is easy to verify that the map is well defined and is projective on the fibers. Furthermore it follows from the construction that $\Phi([v],z)<\Xi(z^{k-1}_\rho,z^{k+1}_\rho)=\xi^2_{\wedge^k\rho}(z)$. Of course, the map $\Phi$ is $\G$-equivariant. In order to conclude the proof it is enough to prove the commutativity of the diagram
$$\xymatrix{
\mathcal B^k_\rho  \ar[r]^\Phi &\P(\wedge^k\C)\\
\deG\times\deG \ar[u]^{\xi_{\boldsymbol{\cdot}}^k} \ar[ur]_{\phi_{\boldsymbol{\cdot}}^1}        &}$$
where, as usual, we denote by $\xi_{\boldsymbol{\cdot}}^k:\deG\times\deG\to\mathcal B_\rho^k$ the parametrization of the foliated limit set for $\rho$, and, to avoid confusion, we denote by $\phi_{\boldsymbol{\cdot}}^1:\deG\times\deG\to \P(\wedge^k\C)$ the parametrization of the foliated limit set for $\wedge^k\rho$ so that $\phi_z^1(y)=\Psi(y^{d-k})\cap \Xi(z^{k-1},z^{k+1})$. It is then a direct computation to check using the definition of the maps $\Psi$ and $\Xi$ that 
if $y^{d-k}\cap z^{k+1}=\langle v\rangle$, so that $\xi^k_z(y)=[v]$,  then 
$$\Phi\circ\xi_z^k(y)=z_1\wedge\ldots\wedge z_{k-1}\wedge v= \phi^1_z(y),$$
which concludes the proof.

\end{proof}
\section{Lebesgue measure}

The goal of the section is to prove the following theorem. 

\begin{thm}\label{thm:Leb0}
Let $\Gamma$ be a hyperbolic group whose boundary is not a sphere. Let $\rho:\Gamma\to{\rm PSL}(d,\mb{C})$ be a $k$-hyperconvex representation. Then for every $z\in\partial\Gamma$ the Lebesgue measure of $\xi_z^k(\partial\Gamma)$ vanishes.
\end{thm}

Abusing notation in this section, we tacitly identify $B_\rho^k$ with  $\CP\times \deG$ via $L$, and drop the $L$ from the $\PSL(2,\C)$-cocycle defined in \eqref{eqn: trivial bundle cocycle}.
Thus, we write 
\[\gamma : (x,t) \mapsto (\rho^k(\gamma,t)x, \gamma t) \in \CP\times \deG\]
for the $\rho$-action on $\mathcal B_\rho^k$ in the coordiantes afforded by $L$.  

Following Ahlfors \cite{Ahlfors66} (see also \cite[Theorem 4.134]{Kap:book} or \cite[Theorem 3.18]{MatTa}),
we study $\G$-equivariant fiberwise harmonic extensions of the characteristic function of  the foliated limit set (Definition \ref{def: double limit set}) in $\CP\times\deG$ to the bundle $\HH^3\times\deG$ defined in Section \ref{sec:bundle}.  

Key in the study of $\G$-equivariant harmonic functions on $\HH^3$ is the family of \emph{harmonic} or \emph{visual} measures $\{\nu_x\}_{x\in \H^3}$ on $\partial \H^3$.
The visual measure $\nu_x$ is obtained by pushing forward the rotationally symmetric solid angle measure on $T_x^1\H^3$ to $\partial \H^3$ via the exponential map.
This family of measures satisfies that if $f: \partial\H^3\to \R$ is $\nu_{x}$-integrable, then 
\begin{equation}\label{eqn: measure equivariance}
    \int f ~d\nu_{x} = \int f\circ \gamma\inverse~ d\nu_{\gamma x}
\end{equation}
for all $\gamma \in \PSL(2,\C)$ (see, e.g., \cite[\S3.12]{Kap:book}).

There is an extension operator from $L^1(\partial \H^3)$\footnote{The density $\frac{d\nu_x}{d\nu_y}$ is bounded, so the $\nu_x$ and $\nu_y$ integrable functions on $\partial\H^3$ coincide,  for all $x$ and $y$, and we denote the integrable functions by $L^1(\partial \H^3)$.} to harmonic functions on $\HH^3$.  
For $f\in L^1(\partial \H^3)$, the harmonic function $P[f]: \HH^3\to \R$ defined by 
\[P[f](x) = \int f ~d\nu_x\]
 is continuous along conical sequences converging to Lebesgue density points\footnote{Recall that a Lebesgue density point $p$ of an integrable function $f : \partial \H^3\to \R$ is such that $\frac{1}{\nu(B)}\int_B f ~ d\nu$ converges to $f(p)$ for nice sets $B$ shrinking to $p$.} of $f$  (see \cite[Proposition 3.20]{MatTa} or \cite{Ahlfors:Poisson} for a proof).

\begin{lemma}
    Suppose $f : \partial \H^3 \times \deG \to \R$  satisfies $f(\cdot , t) \in L^1(\partial \H^3)$ for all $t\in \deG$ and $f$ is $\Gamma$-invariant, i.e., $f(\gamma.(\theta, t)) = f(\theta, t)$ for all $\gamma \in \Gamma$.
    Then the function $P[f]: \H^3\times \deG \to \R$ defined by 
    \[P[f](x, t) := \int f(\theta, t) ~d\nu_x(\theta)\]
    is  $\G$-invariant and, for every $t\in\deG$ and every conical sequence $x_n$ converging to a Lebesgue density point $\theta$ of $f(\cdot ,t)$, $P[f](x_n, t)$ converges to $f(\theta,t)$.
\end{lemma}
\begin{proof}
    Convergence along conical sequences to Lebesgue points is easy to verify and follows directly from the classical case.  To show $\G$-invariance, we compute
    \begin{align*}
        P[f](\gamma.(x,t)) & = P[f](\rho^k(\gamma,t)x,\gamma t)\\
        &= \int f(\theta, \gamma t)~d\nu_{\rho^k(\gamma, t)x}(\theta) \\
        &=\int f\circ \gamma\inverse (\theta,\gamma t) ~d\nu_{\rho^k(\gamma ,t)x}(\theta)\\
        &=\int f(\rho^k(\gamma\inverse, \gamma t)\theta,t) ~d\nu_{\rho^k(\gamma ,t)x}(\theta) \\
        &=\int f(\theta, t)~d\nu_x(\theta) = P[f](x,t).
    \end{align*}
    To go from the second to the third line, we used $\G$-invariance of $f$, and to go from the fourth to the fifth line, we used \eqref{eqn: measure equivariance} and the cocycle property \eqref{eqn: mobius cocycle} that $\rho^k(\gamma\inverse, \gamma t) = \rho^k(\gamma, t)\inverse$.
\end{proof}
We remark that, as in the classical case, the function $P[f]$ is also leafwise harmonic, but this won't play a role in our arguments.
Consider now
\begin{equation}\label{e.F}
\begin{array}{cccl}
\Phi:&\HH^3\times\deG&\to&\R\\
&(x,t)&\mapsto&\int_{\partial\mb{H}^3}{\chi_{\Lambda_t}(\theta) ~d\nu_x(\theta).}
\end{array}
\end{equation}
where
\[
\chi_{\Lambda_t}(\theta)=\left\{
\begin{array}{l l}
1 &\text{\rm if $\theta\in\Lambda_t$},\\
0 &\text{\rm if $\theta\not\in\Lambda_t$}\\
\end{array}
\right.
\]
is the characteristic function of the limit set $\Lambda_t=\xi_t^k(\deG)\subset \CP\cong\P(t^{k+1}/t^{k-1})$.

The following proposition summarizes the properties of $\Phi$ established above.
\begin{prop}\label{p.Kernel} Let $\rho:\G\to\PSL(d,\C)$ be $k$-hyperconvex. The function $\Phi$ from Equation \eqref{e.F} satisfies: 
\begin{enumerate}
\item{For every point $\zeta\in\partial\mb{H}^3$ of density for $\Lambda_t$ or $\Lambda_t^c$ and for every sequence $x_n\in\mb{H}^3$ converging to $\zeta$ conically 
we have $\lim_{n\to \infty}\Phi(x_n,t)=\chi_{\Lambda_t}(\zeta)$.}
\item{We have 
\[
\Phi\circ \gamma = \Phi, ~ \forall \gamma \in \G.
\]
}
\end{enumerate}
\end{prop}
We now have all the necessary ingredients to prove Theorem \ref{thm:Leb0}:
\begin{proof}[Proof of Theorem \ref{thm:Leb0}]
Define the projection $\pi:\partial\Gamma^{(3)}\to\mb{H}^3$ that sends the triple $(t,p,q)$ to the orthogonal projection of $\xi_t^k(p)$ to the geodesic $[\xi_t^k(t),\xi_t^k(q)]$. Note that as $\xi_t^k$ is injective and $t,p,q$ are pairwise distinct, the projection is well defined. Furthermore, it has the equivariant property $\pi(\gamma t,\gamma p,\gamma q)=\rho^k(\gamma,t)\pi(t,p,q)$.

Suppose that $\Lambda_t$ has positive Lebesgue measure for some $t\in\partial\Gamma$. Let $\xi_t^k(\theta)\in\Lambda_t$ be a point of density for the Lebesgue measure, i.e., a Lebesgue density point for $\chi_{\Lambda_t}$ with $\chi_{\Lambda_t}(\xi_t^k(\theta)) = 1$. Choose a sequence $(t,\theta_n,\theta)$ with $\theta_n\to\theta$. By the definition of $\pi$, the point $\pi(t,\theta_n,\theta)$ converges to $\xi_t^k(\theta)$ along the geodesic $[\xi_t^k(t),\xi_t^k(\theta)]$, in other words, the convergence is conical. In particular Proposition \ref{p.Kernel} (1) gives
\begin{equation}\label{e.Leb1}
\Phi(\pi(t,\theta_n,\theta),t)\to 1.
\end{equation}

However, as $\Gamma$ acts co-compactly on the space of distinct triples $\partial\Gamma^{(3)}$, we can find a compact fundamental domain $K\subset\partial\Gamma^{(3)}$ and an infinite sequence of elements $\gamma_n$ such that $(t,\theta_n,\theta)\in\gamma_nK$.
Thus, by Property (2), we also have
\begin{align*}
\Phi(\pi(t,\theta_n,\theta),t) &=\Phi({\rho^k}(\gamma_n^{-1}, t)\pi(t,\theta_n,\theta),\gamma_n^{-1}t)\\
&=\Phi(\pi(\gamma_n^{-1}t,\gamma_n^{-1}\theta_n,\gamma_n^{-1}\theta),\gamma_n^{-1}t)\\
 &=\int_{\Lambda_{\gamma_n^{-1}t}} d\nu_{\pi(\gamma_n^{-1}t,\gamma_n^{-1}\theta_n,\gamma_n^{-1}\theta)}.
\end{align*}

By our choice of the sequence $\gamma_n$, the point $(\gamma_n^{-1}t,\gamma_n^{-1}\theta_n,\gamma_n^{-1}\theta)$ belongs to the compact set $K$. Hence, by continuity of $\pi$, the projection $\pi(\gamma_n^{-1}t,\gamma_n^{-1}\theta_n,\gamma_n^{-1}\theta)$ lies in the compact set $\pi(K)\subset\mb{H}^3$. 
In particular we can assume up to extracting a subsequence that $\gamma_n^{-1}t$ converges to a point $s\in\partial\Gamma$, and that the (proper) closed sets $\Lambda_{\gamma_n^{-1}t}$ converge in the Hausdorff topology to the (proper) closed set $\Lambda_s$ (recall Lemma \ref{l.conti}). In particular, as $\Lambda_s$ is not equal to $\mb{CP}^1$, we can find a small disk $D\subset\mb{CP}^1$ entirely contained in the complement of all $\Lambda_{\gamma_n^{-1}t}$ with $n$ sufficiently large. Therefore, we get
\[
\int_{\Lambda_{\gamma_n^{-1}t}}d\nu_{\pi(\gamma_n^{-1}t,\gamma_n^{-1}\theta_n,\gamma_n^{-1}\theta)}
<\sup_{x\in\pi(K)}\int_{\mb{CP}^1\setminus D}d\nu_x<1
\]
where in the last inequality we used Property (2). Thus, we reached a contradiction with Equation \eqref{e.Leb1}. This finishes the proof of the theorem.
\end{proof}

\section{Groups admitting hyperconvex representations}

As a consequence of Proposition \ref{prop:map xi}, we see that if the hyperbolic group $\Gamma$ admits a $k$-hyperconvex representation, then its boundary $\partial\Gamma$ embeds in the 2-sphere $\mb{CP}^1=\mb{P}(z^{k+1}/z^k)$ for any $z\in\partial\Gamma$ via the map $\xi_z^k$. This strongly restricts the class of groups that can admit such representations. In this section we prove Theorem \ref{thmINTRO:Kleinan} which we recall:

\begin{thm}\label{t.kleinian}
Let $\rho:\Gamma\to{\rm PSL}(d,\mb{C})$ be $k$-hyperconvex. Then $\Gamma$ is virtually isomorphic to a convex-cocompact Kleinian group.
\end{thm}

As the proof uses 2-dimensional quasi-conformal analysis, before starting the argument, we briefly recall the relevant definitions and facts.

\subsection{Quasi-conformal and quasi-symmetric maps}
Quasi-conformal homeomorphisms between domains in $\CP$ are discussed in \cite{Ahlfors:lectures} as orientation preserving homeomorphisms that only distort the modulus of a quadrilateral by a bounded multiplicative factor.  The following (equivalent) definition characterizes quasi-conformal homeomorphisms of $\CP$ as those for which the image of a circle is contained in an annulus of bounded modulus. 
\begin{dfn}[Quasi-conformal]
\label{dfn:quasi-conformal}
An orientation preserving homeomorphism $f:\mb{CP}^1\to\mb{CP}^1$ (resp. $f:\mb{H}^2\to\mb{H}^2$) is $K$-quasi-conformal for some constant $K\ge 1$ if in all affine charts it satisfies the following 
\[
\limsup_{r\to 0}{\frac{\sup_{|z-w|=r}{|f(z)-f(w)|}}{\inf_{|z-w|=r}{|f(z)-f(w)|}}}\le K
\]
for every $z$ in the affine chart.

The previous definition is easily seen to be equivalent to the following in terms of cross-ratios: An orientation preserving homeomorphism $f:\mb{CP}^1\to\mb{CP}^1$ is $K$-quasi-conformal if for every $a,b,c,d\in\mb{CP}^1$ with $|\mb{B}(a,b,c,d)|=1$ we have 
\begin{equation}\label{eqn: bounded cross ratio}
    \frac{1}{K}\le|\mb{B}(f(a),f(b),f(c),f(d))|\le K.
\end{equation}

Here, we denote by $\mb B$ the projective cross-ratio, normalized so that for $(z_1,z_2,z_3,z_4)$
\begin{equation}\label{e.cr}\mb B(z_1,z_2,z_3,z_4)=\frac{(z_3-z_1)(z_4-z_2)}{(z_2-z_1)(z_4-z_3)}.
\end{equation}
Equivalently it holds $\mb B(0,1,z,\infty)=z$.
\end{dfn}

\begin{dfn}[Quasi-M\"obius]\label{def: quasi-mobius}
    Let $C\subset \CP$ be a set.
    A map $f: C \to \CP$ is $K$-quasi-M\"obius for some $K\ge 1$ if for all $a, b, c, d \in C$ with $|\mb B(a,b,c,d)| = 1$, \eqref{eqn: bounded cross ratio} holds.
\end{dfn}

In particular, the restriction of a $K$-quasi-conformal homeomorphism $f: \CP\to \CP$ to any subset $C\subset\CP$ is $K$-quasi-M\"obius.

{Every $K$-quasi-conformal map $f:\mb{H}^2\to\mb{H}^2$ extends to $\partial\mb{H}^2=\mb{RP}^1$ and the restriction to the boundary is a so-called quasi-symmetric homeomorphism.

\begin{dfn}[Quasi-symmetric]
\label{dfn:quasi-symmetric}
An orientation preserving homeomorphism $f:\mb{RP}^1\to\mb{RP}^1$ is $K$-quasi-symmetric for some constant $K\ge 1$ if for all 4-tuples of distinct points $z_1,z_2,z_3,z_4\in\mb{RP}^1$ with $\mb{B}(z_1,z_2,z_3,z_4)=-1$ we have
\[
\frac{1}{K}\le\mb{B}(f(z_1),f(z_2),f(z_3),f(z_4))\le K.
\]
\end{dfn}

Conversely, by a result of Ahlfors-Beurling \cite{AB:qcextension}, every quasi-symmetric homeomorphisms of $\dH$ can be continuously extended to quasi-conformal homeomorphisms of $\bH^2$:
\begin{thm}[{Ahlfors-Beurling \cite{AB:qcextension}}]\label{thm:DEextension}
For every $K$ there exists $K'$ such that every $K$-quasi-symmetric homeomorphism of $\mb{RP}^1$ is the boundary restriction of a $K'$-quasi-conformal homeomorphism of $\mb{H}^2$.
\end{thm}
}

\subsection{Proof of Theorem \ref{t.kleinian}}
We will split the proof into several cases.

\subsubsection{Case 1: $\deG=\mb{CP}^1$}\label{ss.sphere}\

We consider, for every $t\in\deG$, the tangent projection 
$$\xi_t:=\xi^k_t:\partial\Gamma\to\P(t^{k+1}/t^{k-1})=\CP$$
defined in Section \ref{s.projmap}, and omit, for the rest of the section, $k$ from the notation for the sake of readability.
It follows from Proposition \ref{prop:map xi}  that, for every $t\in\partial\Gamma$, the map $\xi_t$ is a homeomorphism being a continuous bijection between compact Hausdorff spaces. 
 
 We fix $z\in\deG$ and use the marking $\xi_z:\deG\to\mathbb P(z^{k+1}/z^{k-1})$  to push forward the standard action of $\G$ on $\deG$ to an action $\rho_z$ of $\G$ on $\CP=\mathbb P(z^{k+1}/z^{k-1})$. More precisely the action $\rho_z$ is  given by 
 \begin{equation}\label{e.rhot}\rho_z(\g)=\xi_z\circ\g\circ(\xi_z)^{-1}.
 \end{equation}
We will show that such action is by uniformly quasi-conformal homeomorphisms. This means that there exits a constant $K$ such that for all $\g\in\G$, $\rho_z(\g)$ is $K$-quasi-conformal. By the following classical result of Sullivan \cite{Sul81}, this implies that $\rho_z(\Gamma)$ is (quasi-conformally) conjugate to a co-compact Kleinian group.
\begin{thm}[Sullivan, see e.g. {\cite[Theorem 23.3]{DrutuKapovich}}]\label{t.Sullivan}
A uniformly quasi-conformal subgroup of $\Homeo(\CP)$ is quasi-conformally conjugated to a subgroup of $\PSL_2(\C)$.
\end{thm}

When $\rho:\G\to\PSL(d,\C)$ is $k$-hyperconvex, for every $t\in\deG$ we consider the cross-ratio $\mb B_t$ on $\deG^4$ given by
$$\mb B_t(a,b,c,d)=\mb B(\xi_t(a),\xi_t(b),\xi_t(c),\xi_t(d)).$$
\begin{lem}\label{lem: action on 6-tuples} Let $\rho:\G\to\PSL(d,\C)$ be $k$-hyperconvex.
The diagonal action of $\Gamma$ on the space
\[
\cal B:=\{(a,b,c,d,x,y)\in\partial\Gamma^6|\, (a,b,d)\in\partial\G^{(3)}, |\mb B_x(a,b,c,d)|=1\}.
\]
is properly discontinuous and co-compact.
\end{lem}

\begin{proof}
Recall from \eqref{eqn: mobius cocycle} that for every $\g\in\G$ and $t\in\deG$, $\rho(\g)$ induces a projective map
$$\rho(\g,t):\mathbb P(t^{k+1}/t^{k-1})\to\mathbb P((\gamma t)^{k+1}/(\gamma t)^{k-1}),$$ 
so that for every $a\in\deG$, $\xi_{\gamma t}(\gamma a)=\rho(\g,t)\xi_{t}(a)$. Since the cross-ratio $\mb B$ is invariant under projective maps,  it follows  that for every $\g\in\G$ and every $(a,b,c,d,x)\in\deG^5$ with $(a,b,d)\in\G^{(3)}$,
$$\mb B_x(a,b,c,d)=\mb B_{\g x}(\g a,\g b,\g c,\gamma d)$$
and thus $\G$ acts on $\cal B$.

The projection
\[
\begin{array}{ccc}\cal B&\to &\partial\Gamma^{(3)}\\
(a,b,c,d,x,y)&\mapsto&(a,b,d).
\end{array}
\]
is $\Gamma$-equivariant and has fibers homeomorphic to $\mb S^1\times\partial\Gamma\times\partial\Gamma$, thus in particular compact. As the action $\Gamma\curvearrowright\partial\Gamma^{(3)}$ is properly discontinuous and co-compact, the same is true for the action of $\Gamma$ on $\mc{B}$.
\end{proof}

\begin{prop}\label{l.uqc} Let $\rho:\G\to\PSL(d,\C)$ be $k$-hyperconvex. There exists a $K$ such that, for every  $x,y
\in\deG$, the homeomorphism 
\[
\xi_y\circ\xi_x^{-1}: \xi_x(\deG) \to \xi_y(\deG)
\]
is $K$-quasi-M\"obius.
\end{prop}

\begin{proof}
Let $\cal B$ be the set considered in Lemma \ref{lem: action on 6-tuples}. We consider the continuous function 
\[
\begin{array}{cccc}f:&\cal B&\to &\R\\
&(a,b,c,d,x,y)&\mapsto&|\mb B_y(a,b,c,d)|.
\end{array}
\]
It follows from the same arguments as in the proof of Lemma \ref{lem: action on 6-tuples} that $f$ is $\G$-invariant and thus descends to a continuous function on the compact space $\cal B/\G$. Since, by definition, for any $(a,b,c,d,x,y)\in\cal B$, $c$ is  distinct from $a$ and $b$ is distinct from $d$, the function $f$ never vanishes, and thus, by compactness of $\cal B/\G$, it is uniformly bounded away from zero and infinity. This implies that the homeomorphisms $\xi_y\circ \xi_x^{-1}$ are uniformly quasi-M\"obius, and concludes the proof.
\end{proof}
Note that for $\deG = \CP$, Proposition \ref{l.uqc} states that 
\[
\xi_y\circ\xi_x^{-1}:\mb{P}(x^{k+1}/x^{k-1})\to\mb{P}(y^{k+1}/y^{k-1})
\]
is a $K$-quasi-conformal homeomorphism.

\begin{prop}\label{p.uqc1} Suppose $\deG = \CP$.
The action $\rho_z : \G \to \Homeo(\mb{P}(z^{k+1}/z^{k-1}))$ defined in Equation \eqref{e.rhot} is uniformly quasi-conformal.
\end{prop}

\begin{proof}
For every $\gamma\in\Gamma$ and $x\in\deG$ we have 
\begin{align*}
\rho_z(\gamma)\xi_z(x) &=\xi_z(\gamma x)\\
&=\rho(\gamma,\g^{-1}z)\xi_{\gamma^{-1}z}(x)\\
&=\rho(\gamma,\g^{-1}z)\circ(\xi_{\gamma^{-1}z}\circ\xi_z^{-1})\circ\xi_z(x).
\end{align*}
Since $\rho(\gamma,\g^{-1}z): \mb{P}((\g^{-1}z)^{k+1}/(\g^{-1}z)^{k-1})\to \mb{P}(z^{k+1}/z^{k-1})$ is conformal and by Proposition \ref{l.uqc} the homeomorphism $\xi_{\gamma^{-1}z}\circ\xi_z^{-1}$ is $K$-quasi-conformal for $K$ independent on $\g\in\G$ and $z\in\deG$, the conclusion follows.
\end{proof}

Thus we have proven that $\Gamma$ admits a uniformly quasi-conformal action $\rho_z$ on $\mb{CP}^1$.  
Since the action $\rho_z$ is induced by the action of the hyperbolic group $\G$ on its boundary, the kernel of $\rho_z$ is finite and the action is discrete.
Thus, by Theorem \ref{t.Sullivan},  the image $\rho_z(\G)\subset\Homeo(\mb{P}(z^{k+1}/z^{k-1}))$ is quasi-conformally conjugated to a Kleinian group.
This shows that
$\Gamma$ is isomorphic to a finite extension of a Kleinian group and concludes the proof of Theorem \ref{t.kleinian} in the case of groups with sphere boundaries.
\qedhere

\subsubsection{Case 2: $\deG\subset\mb{CP}^1$ is a Sierpinski carpet or a circle}\label{ss.Sierpinski} 

Recall that a {\em Sierpinski carpet} is the complement in $\mb{CP}^1$ of a the interiors of a collection of pairwise disjoint disks. The boundaries of the disks provide a collection of disjoint Jordan curves in the carpet, the so-called {\em peripheral circles}. These can be topologically characterized as those Jordan curves that do not disconnect the carpet, and in particular they do not depend on the embedding. 

When the Sierpinski carpet is the boundary of a hyperbolic group we furthermore have:
\begin{thm}[{\cite[Theorem 8]{KK00}}]\label{t.kk}
Let $\G$ be a hyperbolic group with boundary homeomorphic to the Sierpinski carpet. Then:
\begin{enumerate}
\item{There are only finitely many $\Gamma$-orbits of peripheral circles.}
\item{The stabilizer ${\rm Stab}_\Gamma(C)$ of a peripheral circle $C\subset\partial\Gamma$ is a quasi-convex surface subgroup acting co-compactly on the set $C^{(3)}$ of distinct triples in $C$.} 
\end{enumerate} 
\end{thm}
In order to prove Theorem \ref{t.kleinian} for groups whose boundary is a Sierpinski carpet we will choose $z\in\deG$ and will extend the natural action $\rho_z:\G\to\Homeo(\xi_z(\deG))$, defined as in Equation \eqref{e.rhot}, to a uniformly quasi-conformal action $$\hat\rho_z:\G\to\Homeo (\mb{P}(z^{k+1}/z^{k-1})).$$

The first step of the proof is to show, in analogy to Proposition \ref{l.uqc}, that the restriction of $\rho_z$ to the image under $\xi_z$ of the peripheral circles is uniformly quasi-symmetric. More precisely, for every peripheral circle $C\subset\deG$ and every $x\in\deG$, we denote by  $D_x=D_x^C\subset \mb{P}(x^{k+1}/x^{k-1})$  the disk with boundary $\partial D_x^C=\xi_x(C)$. We choose a  uniformization $u_x^D:D_x\to\mb{H}^2$, and denote by 
$u_x^C:\xi_x(C)\to\RP=\partial\mb{H}^2$ its boundary restriction. Observe that $u_x^D$ and $u_x^C$ are well defined up to postcomposition by an element in $\PSL(2,\R)$.
{In the next three results we include the case in which the boundary is homemorphic to a circle, which we understand as a degenerate Sierpinski carpet, with a single peripheral circle equal to the whole set.} 

\begin{lemma}\label{l.calQ} The group $\Gamma$ acts properly discontinuously and co-compactly on the set
\[
\mc{Q}=\left\{(a,b,c,d,x,y)\in\partial\Gamma^6\left|\;
\begin{array}{l}
(a,b,c,d)\in C^{4}\text{ for some peripheral }C\subset\deG\\ 
(a,b,d)\in C^{(3)}\\ 

\mb{B}(u^C_x(\xi_x(a)),u^C_x(\xi_x(b)),u^C_x(\xi_x(c)),u^C_x(\xi_x(d)))=-1\\
\end{array}
\right.\right\}.
\]
\end{lemma}
\begin{proof}
Observe first that $\G$ acts on the set $\mc Q$.  Indeed $\xi_{\gamma x}(\gamma a)=\rho(\g,x)\xi_{x}(a)$ for all $a\in\deG$ and $\gamma \in \Gamma$, and $\rho(\g,x)$ is conformal. 
We then have, for every $s\in\deG$
\begin{equation}\label{e.4.23}
\begin{array}{rl}
u^{\g C}_{\g x}(\xi_{\g x}(\g s))=&u^{\g C}_{\g x}(\rho(\g,x)\xi_{ x}( s))\\
=&(u^{\g C}_{\g x}\rho(\g,x)(u^C_{x})^{-1})\circ u^C_x(\xi_{ x}( s)).
\end{array}
\end{equation}
The composition $u^{\g C}_{\g x}\rho(\g,x)(u^C_{x})^{-1}$ is the extension to $\deH^2$ of the  map $$u^{\g D}_{\g x}\rho(\g,x)(u^D_{x})^{-1}:\HH^2\to\HH^2$$
which is holomorphic being composition of holomorphic maps, and thus belongs to $\PSL(2,\R)$. This implies that $u^{\g C}_{\g x}\rho(\g,x)(u^C_{x})^{-1}$ preserves the cross-ratio on $\deH$, which implies our first claim.

The continuous $\Gamma$-equivariant projection 
\[
\begin{array}{ccc}
     \mc{Q}&\to&\displaystyle{\bigsqcup_{C\text{ peripheral }}{C^{(3)}}}  \\
     (a,b,c,d,x,y)&\mapsto &(a,b,d). 
\end{array}
\]
is proper since its fibers are homeomorphic to the compact space $\partial\Gamma\times\partial\Gamma$. Since the quotient of $\bigsqcup_{C\text{ peripheral }}{C^{(3)}}$ by the action of $\Gamma$ is compact (Theorem \ref{t.kk}), the result follows.
\end{proof}

We denote by $\calC\subset\deG$ the (disjoint) union of all peripheral circles; clearly $\G$ acts on $\calC$. We fix $z\in\deG$ and endow every peripheral circle $C\subset \calC$ with the $\RP$-structure induced by the uniformization $u_z^C$.

\begin{prop}\label{p.quasicircles}
For every $z\in\partial\Gamma$ the Jordan curves in $\xi_z(\mc{C})$ are all uniform quasi-circles.
\end{prop}

\begin{proof}
We use Ahlfors' Criterion \cite[Theorem 1]{A63} which says that a (oriented) Jordan curve in $\mb{CP}^1$ is a quasi-circle if and only if $\left|\mb{B}(a,c,b,d)\right|$ is uniformly bounded away from $\infty$ when $(a,b,c,d)$ is positively oriented quadruple on the curve. In order to apply this criterion to the family of Jordan curves $\xi_z(\mc{C})$ we proceed as follows. We consider the space
\[
\mc{R}:=\left\{(z,a,b,c,d)\in\partial\Gamma\times\left(\sqcup_{C\in\mc{C}}{C^{(3)}\times C}\right)\left|\,a<b<c\le d\le a\right.\right\}
\]
and observe that it $\Gamma$-equivariantly fibers over the base
\[
\sqcup_{C\in\mc{C}}{C^{(3)}}
\]
where the fiber over $(a,b,c)$ is the compact set $\partial\Gamma\times\{d\in C\left|d\in[c,a]\right.\}$ (here $[c,a]\subset C$ is the arc with endpoints $c,a$ not containing $b$). As $\Gamma$ acts cocompactly on the base and the fiber is compact, it also acts cocompactly on $\mc{R}$.

Next, we introduce the function 
\[
(z,a,b,c,d)\in\mc{R}\to|\mb{B}_z(a,b,c,d)|
\]
where, as always, $\mb{B}_z$ is the complex cross-ratio defined by 
\[
\mb{B}_z(a,c,b,d)=\frac{\xi_z(b)-\xi_z(a)}{\xi_z(c)-\xi_z(a)}\cdot\frac{\xi_z(d)-\xi_z(c)}{\xi_z(d)-\xi_z(b)}.
\]
In particular, since $\xi_z(d)\neq\xi_z(b)$ (recall that $d$ lies in the arc $[c,a]$ not containing $b$ and $\xi_z$ is injective), we have $\mb{B}_z(a,c,b,d)\in\mb{C}$.

By construction this function is $\Gamma$-invariant and continuous. Since the quotient $\mc{R}/\Gamma$ is compact, we deduce that there exists $K>1$ such that
\[
|\mb{B}_z(a,b,c,d)|<K
\]
for every $(z,a,b,c,d)\in\mc{R}$. Via Ahlfors' Criterion, this shows that $\xi_z(\mc{C})=\sqcup_{C\in\mc{C}}{\xi_z(C)}$ is a collection of uniform quasicircles for every $z\in\partial\Gamma$.
\end{proof}

\begin{prop}\label{p.uqsym}
For every $z\in\deG$, the action $\rho_z:\G\to\Homeo(\xi_z(\calC))$ is uniformly quasi-symmetric.
\end{prop}

\begin{proof}
We consider the $\Gamma$-invariant and continuous function 
\[
\begin{array}{ccl}
\mc{Q}  &\to &\mb{R}  \\
(a,b,c,d,x,y) &\mapsto& 
\mb{B}(u_y^C\xi_y(a),u_y^C\xi_y(c),u_y^C\xi_y(b),u_y^C\xi_y(d)).
\end{array}
\]
Since $\G$ acts co-compactly on $\mc Q$ (Lemma \ref{l.calQ}), the image is uniformly bounded away from zero and $\infty$. 
Thus the map 
\[
(u_y^C\xi_y)(u_x^C\xi_x)^{-1}:\partial\mb{H}^2\to\partial\mb{H}^2
\]
is uniformly quasi-symmetric. 

In order to conclude we need to show that, for every $C\in\calC$ and $\g\in\G$, the map 
$$u_z^{\g C}\rho_z(\g) (u_z^C)^{-1}:\deH\to\deH$$
is uniformly quasi-symmetric. 
To this aim, analogously to the proof of Proposition \ref{p.uqc1}, we denote by $\rho(\g,z):\mb{P}(z^{k+1}/z^{k-1})\to \mb{P}((\g z)^{k+1}/(\g z)^{k-1})$ the conformal map induced by $\rho(\g)$, so that it holds
$$\rho_z(\g)=\xi_z\xi_{\g z}^{-1}\rho(\g,z).$$
We can then write 
$$\begin{array}{rl}
 u_z^{\g C}\rho_z(\g) (u_z^C)^{-1}&=u_z^{\g C}\xi_z\xi_{\g z}^{-1}\rho(\g,z) (u_z^C)^{-1}\\
&=\displaystyle{u_z^{\g C}\xi_z\xi_{\g z}^{-1}(u_{\g z}^{\g C})^{-1}u_{\g z}^{\g C}\rho(\g,z) (u_z^C)^{-1}}\\
&=(u_z^{\g C}\xi_z)(u_{\g z}^{\g C}\xi_{\g z})^{-1}\circ (u_{\g z}^{\g C}\rho(\g,z) (u_z^C)^{-1}).
\end{array}$$
This is the composition of the uniformly quasi-symmetric map $(u_z^{\g C}\xi_z)(u_{\g z}^{\g C}\xi_{\g z})^{-1}$ and the map $(u_{\g z}^{\g C}\rho(\g,z) (u_z^C)^{-1})$ which preserves the cross-ratio (recall Equation \eqref{e.4.23}), and is thus uniformly quasi-symmetric.
\end{proof}
\begin{cor}\label{cor.stab}
For every peripheral circle $C$, and every $z\in\deG$ the action 
$${\rho_z}|_C:{\rm Stab}_\G(C)\to\Homeo(\xi_z(C))$$
is uniformly quasi-symmetric. 
\end{cor}

We now fix a point $z\in\deG$ and  representatives $C_1,\ldots, C_n$ of the peripheral circles modulo the $\G$-action, and denote by $D_j:=D_z^{C_j}\subset \mb{P}(z^{j+1}/z^{j-1})$ the disk bounded by $\xi_z(C_j)$. The second step of the proof is to extend, for every $j$, the uniformly quasi-symmetric action { ${\rho_z}|_{C_j}$ to a uniformly quasi-conformal action $\hat\rho^{D_j}_z:{{\rm Stab}_\G(C_j)}\to \Homeo (D_j).$} For this we will use the following result by Markovic.

\begin{thm}[{\cite[Theorem 1.1]{Markovic}}]\label{t.Markovic}
Let $G$ be a discrete $K$-quasi-symmetric group. Then there exists a $K'$-quasi-conformal map $\phi:\HH^2\to \HH^2$ and a subgroup $T<\PSL(2,\R)$ such that $G = \phi T\phi^{-1}$.
The constant $K'$ is a function of $K$.
\end{thm}
\begin{remark}
    Markovic states his result strictly in terms of a quasi-symmetric conjugation, but the proof actually builds a $K'$-quasi-conformal group extending $G$ as stated in Theorem \ref{t.Markovic} and discussed in the introduction of \cite{Markovic} and in \S8.3, therein.
\end{remark}

\begin{prop}\label{p.hatrz}
 For every $1\leq j\leq n$,  the action $\rho_z|_{C_j}:{{\rm Stab}_\G(C_j)}\to\Homeo(\xi_z(C_j))$ extends to a uniformly quasi-conformal action 
 $$\hat\rho^{D_j}_z:{{\rm Stab}_\G(C_j)}\to \Homeo (D_j).$$ 
\end{prop}

\begin{proof}
This follows combining Theorem \ref{t.Markovic} and Corollary \ref{cor.stab}.
\end{proof}

We will use the actions $\hat\rho_z^{D_j}$ to extend the action $\rho_z:\G\to \Homeo(\xi_z(\deG))$ to an action by homeomorphisms on $\P(z^{k+1}/z^{k-1})$. In order to check that the extension is uniformly quasi-conformal we will use the following result by Bonk 
\begin{prop}[{\cite[Proposition 5.1]{Bonk}}]\label{prop.Bonk}
Suppose that $\{D_i|\, i\in I\}$ is a non-empty family of pairwise disjoint closed Jordan regions in $\CP$,  let $$f : \CP \setminus \bigcup_{i\in I} {\rm int}(D_i) \to \CP$$ be  $K$-quasi-M\"obius. If for $i\in I$ the Jordan curves $S_i := \partial D_i$ are $K$-quasicircles, and for every $i$ there exists a $K$-quasi-conformal map $F_i:D_i\to \CP$ with $F_i|_{S_i}=f$. Then the extension $f\cup F_i$ is $H$-quasi-conformal, where $H$ depends on $K$ only.
\end{prop}
\begin{proof}
While the statement of \cite[Proposition 5.1]{Bonk} is more general, the last step of its proof, carried out at p. 588-589 gives an elementary proof of the statement.
\end{proof}

We now have all the ingredients necessary to obtain the sought uniform quasi-conformal action of $\G$ on $\CP$.
\begin{prop}\label{p.extSier} {Assume that the boundary of $\G$ is a Sierpinski carpet or a circle. }The action $\rho_z$ on $\xi_z(\bord\G)$ extends to a uniformly quasi-conformal action $\hat\rho_z$ on $\P(z^{k+1}/z^{k-1})$.
\end{prop}
\begin{proof}
By the structure of a Sierpinski carpet, the complement 
\[\mb{P}(z^{k+1}/z^{k-1})\setminus\xi_z(\partial\Gamma)\] 
is a disjoint union of disks $D$ bounded by the images of the peripheral circles $\xi_z(C)$.  We fix, as above, representatives $C_1,\cdots,C_n$ for the $\G$-orbits on $\calC$,  and denote by $D_j\subset \mb{P}(z^{k+1}/z^{k-1})$ the complementary disk with boundary $\xi_z(C_j)$. We furthermore choose representatives $\alpha_C$ for the cosets of ${\rm Stab}_\Gamma(C_i)$, $1\leq i\leq n$, so that  every complementary disk $D$, bounded by the image $\xi_z(C)$ of the peripheral circle $C$, can be written uniquely as $\alpha_C D_j$ for some $j$. The homeomorphism $$\rho_z(\alpha_C):\xi_z(C_j)\to\xi_z(C)$$ is $K$-quasi-symmetric (Proposition \ref{p.uqsym}), and we fix, for every $D$, a $K^*$-quasi-conformal homeomorphism 
$$f_D:D_j\to D$$
extending $\rho_z(\alpha_C)$, e.g., via Theorem \ref{thm:DEextension}. Here $K^*$ only depends on $K$.

Now we extend the action element by element and disk by disk: for every $\g\in\G$ and every complementary disk $D=\alpha_C D_j$ we set
\begin{equation}\label{e.rhohD}\hat\rho^D_z(\g)= f_{\g D}\circ \hat\rho^{D_j}_z(\alpha^{-1}_{\g C}\g\alpha_C)\circ f^{-1}_D:D\to\g D.
\end{equation}
Here $\hat\rho^{D_j}_z$ is the {$(K')^2$}-quasi-conformal action of $\Stab_\G(C_j)$ on $D_j$ studied in Proposition \ref{p.hatrz} (by construction $\alpha^{-1}_{\g C}\g\alpha_C\in\Stab_\G(C_j)$), which extends $\rho_z$ on $\xi_z(C_j)$. Here, with a slight abuse, we denote by $\g D$ the complementary disk bounded by $\xi_z(\g C)$, and $K'$, again, only depends on $K$.

Since $\hat \rho^D_z(\g)$ extends $\rho_z(\g)|_C$, we obtain through this construction a well defined element $\hat\rho_z(\g)\in \Homeo(\mb{P}(z^{k+1}/z^{k-1}))$. For every $\g\in \G$, $\hat \rho^D_z(\g)$ is furthermore $(K^*K')^2$-quasi-conformal: indeed it is so when restricted to any complementary disk $D$ by Proposition \ref{p.hatrz} and Equation \eqref{e.rhohD}. 
Since $\rho_z(\g)|_{\xi_z(\deG)}$ is uniformly quasi-Möbius (Proposition \ref{l.uqc}), and the elements of $\calC$ are uniform quasi-circles (Proposition \ref{p.quasicircles}), it follows from Proposition \ref{prop.Bonk} that the extension $\hat\rho_z(\g)$ is uniformly quasi-conformal.

In order to conclude the proof, it is enough to verify that $\hat\rho_z:\G\to\Homeo(\P(z^{k+1}/z^{k-1}))$ is a group homomorphism. Clearly $\hat\rho_z(\Id)=\Id$.
Let us then fix elements $\beta,\g\in\G$ and  a complementary disk $D\subset \mb{P}(z^{k+1}/z^{k-1})$ with boundary $\xi_z(C)$ for a peripheral circle $C$. Let $1\leq j\leq n$ and $\alpha_C, \alpha_{\gamma C}, \alpha_{\beta\gamma C}$ among our representatives of the cosets of $\Stab_\G(C_j)$ so that $C=\alpha_C C_j$, $\gamma C=\alpha_{\g C} C_j$ and $\beta\gamma C=\alpha_{\beta\g C} C_j$.
By definition 
$$\hat\rho_z^D(\beta\gamma)=f_{\beta\gamma D}\circ\hat\rho^{D_j}_z(\alpha_{\beta\gamma C}^{-1}\beta\gamma\alpha_{C})\circ f_D^{-1}$$
where $f_{\beta\gamma D}:D_j\to \beta\gamma D$ (resp. $f_{D}:D_j\to  D$) is our chosen $K^*$-quasi-conformal extension of $\rho_z(\alpha_{\beta\gamma C})|_{\xi_z(C_j)}$ (resp. of $\rho_z(\alpha_{C})|_{\xi_z(C_j)}$).

We can write 
$$\alpha_{\beta\gamma C}^{-1}\beta\gamma\alpha_{C}=(\alpha_{\beta\gamma C}^{-1}\beta\alpha_{\gamma C})\cdot(\alpha_{\gamma C}^{-1}\gamma\alpha_{C}).$$
Since both factors belong to ${\rm Stab}_\Gamma(C_j)$, we deduce
$$\hat\rho^{D_j}_z(\alpha_{\beta\gamma C}^{-1}\beta\gamma\alpha_{C})=\hat\rho^{D_j}_z(\alpha_{\beta\gamma C}^{-1}\beta\alpha_{\g C})\circ\hat\rho^{D_j}_z(\alpha_{\gamma C}^{-1}\gamma\alpha_{C}).$$
From this it directly follows 
$$\begin{array}{rl}
\hat\rho_z^D(\beta\gamma)&=f_{\beta\gamma D}\circ\hat\rho^{D_j}_z(\alpha_{\beta\gamma C}^{-1}\beta\gamma\alpha_{C})\circ f_D^{-1}\\
&=f_{\beta\gamma D}\circ\hat\rho^{D_j}_z(\alpha_{\beta\gamma C}^{-1}\beta\alpha_{\g C})\circ\hat\rho^{D_j}_z(\alpha_{\gamma C}^{-1}\gamma\alpha_{C})\circ f_D^{-1}\\
&=(f_{\beta\gamma D}\circ\hat\rho^{D_j}_z(\alpha_{\beta\gamma C}^{-1}\beta\alpha_{\g C})\circ f_{\g D}^{-1})\circ(f_{\g D}\circ\hat\rho^{D_j}_z(\alpha_{\gamma C}^{-1}\gamma\alpha_{C})\circ f_D^{-1})\\
&=\hat\rho_z^{\g D}(\beta)\circ\hat\rho_z^D(\gamma)
\end{array}$$
which concludes the proof in the case of Sierpinski carpet boundary. 

The same argument applies in the case where the boundary is a circle, with the simplification that in this case there are only two disks to fill.
\end{proof}

Theorem \ref{t.kleinian} for groups $\G$ whose boundary is a Sierpinski carpet then follows from Theorem \ref{t.Sullivan} as in the previous case, since we have produced a discrete uniformly quasi-conformal action of any such group $\G$ with finite kernel.

\subsubsection{General case}
In order to deduce the general case  from the previous two cases we use a characterization of  Haïssinsky of groups virtually isomorphic to Kleinian groups based on the Ahlfors regular conformal dimension. The \emph{Ahlfors regular conformal dimension} $\cdAR(\G)$ of a hyperbolic group $\G$ is a numerical invariant of the quasi-isometry class of $\G$. We will not need its definition (see \cite[Section 1.1.3]{Ha15}) but the following theorem by Haïssinsky:
\begin{thm}[{\cite[Theorem 1.9]{Ha15}}]\label{t.Hei1}
Let $\G$ be a non-elementary hyperbolic group with planar boundary. Then $\G$ is virtually isomorphic to a convex-cocompact Kleinian group if and only if $\cdAR(\G)<2$.
\end{thm}
A subgroup of a hyperbolic group with planar boundary $\deG$ is a \emph{carpet subgroup} if its boundary is homeomorphic to a Sierpinski carpet. Haïssinski proves
\begin{thm}[{\cite[Theorem 1.13]{Ha15}}]\label{t.Hei2}
 Let $\G$ be a non-elementary hyperbolic group with planar boundary with no elements of order two. The following conditions are equivalent.
\begin{itemize} 
\item $\G$ is virtually isomorphic to a convex-cocompact Kleinian group. 
\item Any quasi-convex carpet subgroups H of $\G$ satisfies $\cdAR(H)<2$.
\end{itemize}
\end{thm}

\begin{proof}[Proof of Theorem \ref{t.kleinian}]
The result holds if $\deG=\CP$ by Case 1, we can thus assume (using Proposition \ref{prop:map xi}) that $\G$ has planar boundary.

Since $\G$ admits an Anosov representation, it is virtually linear, and thus, by Selberg's lemma, it admits a torsion free subgroup of finite index. The $k$-hyperconvex representation $\rho:\G\to\PSL(d,\C)$ restricts to a $k$-hyperconvex representation of any quasi-convex carpet subgroup $H$ of $\G$, which is then virtually Kleinian by Case 2. Theorem \ref{t.Hei1} implies that $\cdAR(H)<2$, thus by Theorem \ref{t.Hei2}, $\G$ is virtually Kleinian.
\end{proof}

\section{Hausdorff dimension of limit sets}
Recall from Section \ref{s.Anosov} that we denote by $\Lambda^{k}_\rho=\xi^{k}(\deG)\subset\Gr_{k}(\C^d)$ the $k$-th limit set of a $k$-Anosov representation. If $\rho$ is $s$-Anosov for all $s\in\underline k=\{k_1,\ldots, k_l\}$ we further denote by $\calF_{\underline k}$ the flag manifold consisting of flags of subspaces of dimension $k_1\ldots, k_l$, and denote by $\Lambda^{\underline k}_\rho\subset \calF_{\underline k}$ the corresponding limit set.

The goal of the section is to prove the following refinement of Theorem \ref{thmINTRO:Hausdorff}: 
\begin{thm}\label{p.Hffub}
Let $\rho:\G\to\PSL(d,\C)$ be  $s$-hyperconvex. Then the two-dimensional Hausdorff measure of $\Lambda_\rho^{{d-s}}$ vanishes, unless $\Gamma$ is virtually isomorphic to a uniform lattice in ${\rm PSL}(2,\mb{C})$. If we additionally assume that the boundary of $\G$ is a Sierpinski carpet or a circle and that $\rho:\G\to\PSL(d,\C)$ is $s$-hyperconvex for all $s\in\{k_1, \ldots, k_l\}$ then 
$${\rm Hff}(\Lambda_\rho^{{\underline t}})<2$$
where $\underline t=\{d-k_1,\ldots, d-k_l\}$.
\end{thm}
The assumption that the boundary of $\G$ is a Sierpinski carpet is needed in the following proposition, which builds on the work in Section \ref{ss.Sierpinski}:
\begin{prop}\label{p.HffLxzero}
Let $\rho:\G\to\PSL(d,\C)$ be  $1$-hyperconvex. If $\deG$ is a Sierpinski carpet or a circle, then  ${\rm Hff}(\Lambda_z^1)<2$ for every $z\in\deG$.
\end{prop}
\begin{proof}

Recall that, by Proposition \ref{p.extSier}, we can extend the action of $\Gamma$ on $\Lambda^1_z$ (induced by conjugating via $\xi_z$ the action of $\Gamma$ on its boundary) to a uniformly quasi-conformal action on $\mb{P}(z^2)$ and that, by {Sullivan's Theorem \ref{t.Sullivan}}, each such action is quasi-conformally conjugate to a Kleinian action. In other words, there exists a Kleinian representation $\kappa_z:\Gamma\to{\rm PSL}(z^2)$ and a quasi-conformal homeomorphism $f:\mb{P}(z^2)\to\mb{P}(z^2)$ such that $f\rho_z=\kappa_zf$. The representation $\kappa_z$ is convex-cocompact as it comes with an equivariant embedding $f\circ\xi_z:\partial\Gamma\to\mb{P}(z^2)$ and it is well-known that this condition is equivalent to being convex-cocompact (see e.g., \cite[Corollary 32.1]{CanaryLectureNotes}).

In particular, $\Lambda_z^1$ is the image under a quasi-conformal map of the limit set of $\kappa_z$ which, by work of Sullivan \cite{Sul84} and Tukia \cite{Tuk84}, has Hausdorff dimension $\delta<2$. By classical estimates of Gehring and Väisälä \cite[Theorem 12]{GV}, the image of a closed set of Hausdorff dimension smaller than 2 under a quasi-conformal map has still Hausdorff dimension smaller than 2.
\end{proof}
In order to deduce the analogous result for the limit set $\Lambda^1_\rho\subset\CPd$ we use the next proposition, adapting the key arguments in \cite[Proposition 7.3]{PSW:HDim_hyperconvex}  to our context. Given $z^1\in\CPd$, we denote by $\Gr_{d-1}^z(\C^d)$ the open set consisting of subspaces transverse to $z^1$, and consider the map
\[\begin{array}{cccc}
\Xi:&\Gr_{d-1}^z(\C^d)&\to&\P(z^2)\\
&y^{d-1}&\mapsto&[y^{d-1}\cap z^{2}].
\end{array}\]
\begin{prop}\label{p.projLip}
Let $\rho:\G\to\PSL(d,\C)$ be $1$-hyperconvex. Choose $z\in\deG$ and  an open set $U\subset\deG\setminus z$ with  compact closure in $\deG\setminus z$. Then  the restriction of the map $\Xi$ to $\xi^{d-1}( U)$ induces a  bi-Lipschitz map:

\[\Xi:\xi^{d-1}( U)\to \Xi(\xi^{d-1}( U))\subset \xi^{1}_{z}(\deG)\subset \P(z^{2}).\]
\end{prop}
\begin{proof}

Any two Riemannian metrics on $\mb{CP}^{d-1}$ are bi-Lipschitz equivalent. We consider on $\mb{CP}^{d-1}$ the Fubini--Study metric induced by the choice of an Hermitian scalar product $h$ on $\C^d$. This is a Riemannian metric  whose associated distance is given by 
$$\cos d_h([x],[y])=\frac{|h(x,y)|}{\|x\|\|y\|},$$
where we denote by $\|\cdot\|$ the norm associated to $h$ so that $\|x\|^2=h(x,x)$. It is easy to verify that the distance $d^*_h$ induced by the dual Hermitian product on $(\mb{CP}^{d-1})^*=\Gr_{d-1}(\C^d)$ agrees with the Hausdorff distance on closed sets of   $\mb{CP}^{d-1}$, namely for $x^{d-1},y^{d-1}\in \Gr_{d-1}(\C^d)$,
$$d^*_h(x^{d-1},y^{d-1})=\max_{v\in x^{d-1}}\min_{w\in y^{d-1}}d_h([v],[w]).$$

Furthermore if $z^2$ is the orthogonal to $x^{d-1}\cap y^{d-1}$ with respect to the  Hermitian product $h$,  it holds
\begin{equation}\label{e.Fubini1}
d^*_h(x^{d-1},y^{d-1})=d_h([x^{d-1}\cap z^2],[y^{d-1}\cap z^2]).
\end{equation}

The symmetric space associated to $\PSL_d(\C)$ parametrizes all possible Hermitian products up to rescaling, and if two scalar products $h_1,h_2$ are at uniformly bounded distance, the identity 
\begin{equation}\label{e.Fubini2}(\mb{CP}^{d-1}, d_{h_1})\to (\mb{CP}^{d-1}, d_{h_2})
\end{equation}
is uniformly bi-Lipschitz.

Let now $\rho:\G\to\PSL_d(\C)$ be 1-hyperconvex, and choose $z$ and $\mathcal U$ as in the statement of the proposition. It follows from the definition of 1-hyperconvexity that there is a lower bound on the distance between 
\[x^{d-1}\cap y^{d-1} \text{ and } z^2,\]
that holds uniformly for all $x$ and $y \in  U$.
In particular for every $x,y\in U$, there exists an Hermitian product $h_{x,y}$ with respect to which   $x^{d-1}\cap y^{d-1}$ and $z^2$ are orthogonal, and such that $h_{x,y}$ lies at uniform (only depending on $\rho$, $z$, and  $ U$) distance from the original Hermitian product $h$. Since by Equation \eqref{e.Fubini1} we have 
$$d^*_{h_{x,y}}(x^{d-1},y^{d-1})=d_{h_{x,y}}(\Xi(x^{d-1}),\Xi(y^{d-1})),$$
the claim follows from Equation \eqref{e.Fubini2}.
 \end{proof}

\begin{proof}[Proof of Theorem \ref{p.Hffub}]
First, we prove that the 2-dimensional Hausdorff measure of the Grassmannian limit set $\Lambda_\rho^{d-s}\subset\Gr_{d-s}(\C^d)$ is zero and that if $\partial\Gamma$ is a Sierpinski carpet then the Hausdorff dimension of $\Lambda_\rho^{d-s}$ is strictly smaller than 2. Then we will show how to combine these facts together with work of Pozzetti-Sambarino \cite{PSamb} and Pozzetti-Sambarino-Wienhard \cite{PSW:HDim_hyperconvex} to get the result about the Hausdorff dimension of the limit set $\Lambda_\rho^{\underline{t}}$ in the partial flag variety $\mc{F}_{\underline{t}}$.

Let us begin with the first claim. Thanks to Proposition \ref{prop:j to 1i}, we can assume without loss of generality that $s=1$. Furthermore, it is enough to find a finite open cover $\mc{U}=\{U_j\}_{j\in J}$ such that the 2-dimensional Hausdorff measures of $U_j\cap\Lambda_\rho^{d-1}$ vanish for all $j\in J$, 
respectively, in the case of Sierpinski carpet boundary, such that the Hausdorff dimension of  $U_j\cap\Lambda_\rho^{d-1}$ is strictly smaller than 2 for all $j\in J$.
We produce such a good covering using Proposition \ref{p.Hffub}:  we can find a finite set of points $z_1,\cdots,z_n$ and an open covering $\mc{U}=\{U_1,\cdots,U_n\}$ of $\Lambda_\rho^{d-1}=\xi^{d-1}(\partial\Gamma)$ such that the tangent projections $\xi^1(U_j)\to\mb{P}(z_j^2)$ are  bi-Lipschitz onto their image.
By Theorem \ref{thm:Leb0}, the Lebesgue measure of the images $\Lambda^{1}_{z_j}\subset\mb{P}(z_j^2)$ of the limit set $\Lambda_\rho^{d-1}$ under the tangent projections vanishes. As the Lebesgue measure on $\mb{P}(z_j^2)$ is the 2-dimensional Hausdorff measure and since bi-Lipschitz maps preserve the property of having zero 2-dimensional Hausdorff measure, the first claim follows. Analogously, if $\partial\Gamma$ is a Sierpinski carpet, then, by Proposition \ref{p.HffLxzero}, we know that the Hausdorff dimension of $\Lambda^1_{z_j}$ is strictly smaller than 2 for every $j\le n$. As bi-Lipschitz maps preserve the Hausdorff dimension, we can conclude that
${\rm Hff}(\Lambda_\rho^{d-1})<2$.
This concludes the first part of the proof. 

We now consider the limit set in the partial flag variety $\Lambda_\rho^{\underline{t}}\subset\mc{F}_{\underline{t}}$ in the case where $\partial\Gamma$ is a Sierpinski carpet. If $\rho$ is $s$-hyperconvex, Pozzetti-Sambarino-Wienhard \cite[Theorem 5.14]{PSW:HDim_hyperconvex} interpret the Hausdorff dimension of the limit set $\Lambda^{d-s}_\rho\subset\Gr_{d-s}(\C^d)$ as the  exponential growth rate of the $(d-s)$-th root for the action on the symmetric space: $h_{\rho}^{\mathrm a_{d-s}}={\rm Hff}(\Lambda_\rho^{d-s})$. We will not need the precise definition of $h_{\rho}^{\mathrm a_{d-s}}$, it will suffice to recall that, under the hyperconvexity assumptions of the theorem, Pozzetti-Sambarino prove  \cite[Corollary 5.13]{PSamb} that the Hausdorff dimension of $\Lambda^{\underline{t}}_\rho$ coincides with the maximal critical exponent: ${\rm Hff}(\Lambda_\rho^{\underline{t}})=\max_{k_j\in\underline k}\{h_{d-k_j}(\rho)\}$. Combining these results with the first part of the proof we conclude that 
\[
{\rm Hff}(\Lambda_\rho^{\underline{t}})=\max_{k_j\in\underline k}\{{\rm Hff}(\Lambda_\rho^{d-k_j})\}<2,
\]
as desired. 
\end{proof}

We conclude the paper with a concrete example in which knowing the strict inequality ${\rm Hff}(\Lambda^{\underline t}_\rho)<2$ is necessary to apply the work  of Dumas--Sanders \cite[Theorems D and E]{DumasSanders} and  deduce properties of the complex geometry of a co-compact domain of discontinuity.
Any  Anosov representation $\rho:\Gamma\to{\rm PSL}(3,\mb{C})$ admits the following co-compact Guichard--Wienhard domain of discontinuity in the full flag space $\mc{F}_{1,2}$:
\[
\Omega_\rho:=\mc{F}_{1,2}\setminus\mc{K}_\rho
\]
where
\begin{align*}
\mc{K}_\rho:=\{(\xi^1(t),H)\in\mc{F}_{1,2}\left|H\in{\rm Gr}_2(\mb{C}^3)\,,\,t\in\partial\Gamma\,,\,\xi^1(t)\subset H\right.\}\\
\cup\{(\ell,\xi^2(t))\in\mc{F}_{1,2}\left|\ell\in{\rm Gr}_1(\mb{C}^3)\,,\,t\in\partial\Gamma\,,\,\ell\subset\xi^2(t)\right.\}.
\end{align*}

We denote by $N_\rho:=\Omega_\rho/\rho(\Gamma)$ the compact quotient.

\begin{cor}
Let $\rho:\Gamma\to{\rm PSL}(3,\mb{C})$ be a \{1,2\}-hyperconvex representation. Assume that $\partial\Gamma$ is a Sierpinski carpet or a circle. Then $N_\rho$ does not admit a Kähler metric and does not admit any non-constant holomorphic map to a hyperbolic Riemann surface.
\end{cor}

\bibliography{references}{}
\bibliographystyle{amsalpha.bst}

\Addresses
\end{document}